\documentclass[11pt,reqno,a4paper]{amsart}
\usepackage[utf8]{inputenc}

\usepackage{
    amsmath,  amsfonts, amssymb,  amsthm,   amscd,
    gensymb,  graphicx, comment,  etoolbox, url,
    booktabs, stackrel, mathtools, enumitem, mathdots,  microtype, lmodern, mathrsfs, graphicx, tikz, longtable, tabularx, float, tikz, pst-node, tikz-cd, multirow, tabularx, amscd,  bm, array, makecell, diagbox, booktabs, ragged2e, caption, subcaption }

\usepackage[all]{xy}
\usepackage{makecell,slashbox}

\usepackage{xcolor}
\usepackage[utf8]{inputenc}
\usepackage{microtype, fullpage, wrapfig,textcomp,mathrsfs,csquotes,fbb}
\usepackage[colorlinks=true, linkcolor=blue, citecolor=blue, urlcolor=blue, breaklinks=true]{hyperref}
\usepackage[capitalise]{cleveref}
\usepackage{todonotes}
\usetikzlibrary{positioning}
\usetikzlibrary{shapes,arrows.meta,calc}

\theoremstyle{plain}
\newtheorem{theorem}{Theorem}[section]

\newtheorem{corollary}[theorem]{Corollary}
\newtheorem{question}[theorem]{Question}
\newtheorem{prop}[theorem]{Proposition}

\newtheorem{lemma}[theorem]{Lemma}

\theoremstyle{definition}
\newtheorem{definition}[theorem]{Definition}

\newtheorem{exmp}[theorem]{Example}
\theoremstyle{remark}

\newtheorem{remark}[theorem]{Remark}

\newcommand{\QQ} {\mathbb Q}

\def\sO{\mathscr{O}}

\def\sU{\mathscr{U}}

\def\sB{\mathscr{B}}
\def\sN{\mathscr{N}}
\def\u#1{\underline{#1}}
\def\sM{\mathscr{M}}
\def\sA{\mathscr{A}}
\def\sE{\mathscr{E}}
\def\sI{\mathscr{I}}
\def\cA{\mathcal{A}}
\def\cB{\mathcal{B}}
\def\cC{\mathcal{C}}
\def\cK{\mathcal{K}}

\def\cB{\mathcal{B}}
\def\cV{\mathcal{V}}
\def\cI{\mathcal{I}}

\def\cU{\mathcal{U}}

\def\H{\text{H}}

\def\ker{\text{ker}}
\def\DGA{\text{DGA}}

\newcommand{\beq} {\begin{equation}}
\newcommand{\eeq} {\end{equation}}

\renewcommand{\Im} {\operatorname{Im}}

\usetikzlibrary{arrows}

\newcolumntype{x}[1]{>{\centering\arraybackslash}p{#1}}

\author{Soumyadip Thandar}
\address{School of Mathematics, Tata Institute of Fundamental Research, Mumbai, India}
\email{stsoumyadip@gmail.com}
\date{\today}

\begin{document}
\begin{center}
{ 
       {\Large \textbf { \sc  $C_{pq}$-Injective Diagrams and a Combination Theorem for Minimal Models   
          }
       }
\\

\medskip

 {\sc Soumyadip Thandar}\\
{\footnotesize School of Mathematics, Tata Institute of Fundamental Research, Mumbai, India}\\

{\footnotesize e-mail: {\it stsoumyadip@gmail.com}}\\
}
\end{center}

\medskip

\begin{center}
  {\sc Abstract}\\
\end{center}
  
We study diagrams of commutative differential graded algebras (DGAs) over the orbit category $\sO_G$ in the context of equivariant rational homotopy theory. For $G = C_{pq}$ with $p, q$ distinct primes, we give necessary conditions for injectivity. We prove a combination-type result: the equivariant wedge of injective diagrams over $\mathcal{O}_{C_p}$ and $\mathcal{O}_{C_q}$ with retract structure maps yields an injective diagram over $\mathcal{O}_{C_{pq}}$ with a level-wise minimal model. As an application, we construct examples of $C_{pq}$-formal spaces.

\hrulefill

{\small \textbf{Keywords:} System of $\DGA$s, Minimal model, Equivariantly formal, Unstable equivariant rational homotopy theory.}

\indent {\small {\bf 2020 Mathematics Subject Classification:} {55P91, 55P62, 16E45, 18G10.}}

\section{Introduction}\label{section:intro}

Rational homotopy theory provides an algebraic approach to studying the homotopy types of simply connected spaces using rational data. Two such spaces are said to have the same \emph{rational homotopy type} if they are connected by a zigzag of maps, each inducing an isomorphism in rational cohomology. Quillen and Sullivan introduced algebraic models for this purpose: Quillen via differential graded Lie algebras~\cite{QR} and Sullivan via commutative differential graded algebras, now known as Sullivan models~\cite{Sul77}. A Quillen equivalence between the homotopy categories of nilpotent spaces and cohomologically $1$-connected $\DGA$s is established in~\cite{BOUSGUG}, providing a formal justification for studying rational homotopy types through algebraic models. This framework has been developed and explored in various contexts, including~\cite{DGMS, HARST, SY82, LUPT, HIRO, SS12}.

We study topological spaces equipped with a $ G$-action, focusing on their \emph{equivariant rational homotopy type}. Specifically, we consider $G$-spaces $X$ such that the fixed point sets, denoted as $X^H$, is simply connected for every subgroup $H \leq G$. The foundational work of Elmendorf~\cite{ELM83} identifies the homotopy category of $G$-spaces with the category of contravariant functors from the orbit category $\sO_G$ (see \Cref{defn:orbit}) to spaces, enabling an algebraic approach to equivariant topology. In this framework, the Bredon cohomology of a $G$-space with constant coefficients corresponds to a diagram of graded algebras over $\sO_G$, and injective diagrams retain essential topological information. Algebraic models for classifying equivariant rational homotopy types were developed for finite $G$ and for $G = S^1$ by Triantafillou~\cite{GT82} and Scull~\cite{LS02}, respectively. These models take the form of functors from $\sO_G$ to cohomologically $1$-connected DGAs, referred to as \emph{systems of DGAs} when injective. Two $G$-spaces are said to share the same \emph{$G$-rational homotopy type} if their fixed point sets have isomorphic rational cohomology along a zigzag of $G$-maps. It is shown in~\cite{GT82,LS02} that this type is determined by the isomorphism class of their minimal system of DGAs. More recently, Santhanam and Thandar~\cite{santhanam2023equivariant} analyzed this classification for $G = C_p$ using cohomology diagrams to capture equivariant rational data.

The equivariant rational homotopy type of a diagram of graded algebras $\cA^\ast$ over the orbit category $\sO_G$ that is viewed as a diagram of $\DGA$s with zero differential, is determined by its minimal model (see \Cref{def:minimal} and \Cref{thm:existance of minimal}). That is, a minimal system of $\DGA$s $\sM$ over $\sO_G$ together with a morphism $\rho\colon \sM \to \cA^\ast$ inducing an isomorphism in cohomology (such maps are often called \emph{quasi-isomorphisms}). However, the cohomology diagram associated to a given $G$-space is not always injective, as illustrated in \Cref{exmp:noninjective}. In such cases, one must consider the \emph{injective envelope} of the diagram of $\DGA$s (see \Cref{thm:fineee}), whose differential is generally nonzero. This leads to more intricate computations compared to the case of diagrams with zero differential.

In \cite{santhanam2023equivariant}, the authors provide a simple criterion to check whether a diagram of $\DGA$s over the orbit category of $C_p$ (a cyclic group of prime order $p$) is injective (see \Cref{prop:inj C_p}).

Motivated by this result, we pose the following question:

\begin{question}\label{question:1}
\emph{Is there an analogous criterion that ensures injectivity for diagrams of $\DGA$s over the orbit category of $C_{pq}$, where $p$ and $q$ are distinct primes?}
\end{question}

Understanding injective objects in this setting is essential as they serve as targets for minimal models and play a central role in classifying equivariant rational homotopy types. We address the question posed above in \Cref{section:injecti}.

In contrast to the $C_p$-case, where surjectivity of the structure maps is sufficient to guarantee injectivity of a diagram of $\DGA$s over $\sO_{C_p}$ (see statement of \Cref{prop:inj C_p}), we show that this is not the case for $C_{pq}$, where $p$ and $q$ are distinct primes. In particular, we prove that if a diagram over $\sO_{C_{pq}}$ is injective, then all of its structure maps must be surjective. However, the converse does not generally hold (see \Cref{exmp: notinjectiveexample}). To address this, we provide a condition: diagrams over $\sO_{C_{pq}}$ that satisfy \emph{Property I} (defined in \Cref{defn: property I}) are always injective (see \Cref{thm: injective C_pq}). 

A minimal system of $\DGA$s over $\sO_G$ that determines the $G$-rational homotopy type of a $G$-simply connected space is built step by step using \emph{elementary extensions} (see \Cref{def:ele}). These are the equivariant analogue of Hirsch extensions, which are used in the non-equivariant case to construct minimal Sullivan algebras (see \cite[Chapter~16.2]{PGJM}). In a Hirsch extension, one adds new generators in a fixed degree to kill certain cohomology classes. In the equivariant setting, however, the construction of an elementary extension requires taking an \emph{injective resolution} of a diagram of vector spaces (also called \emph{associated system of vector spaces} \Cref{def:associ:vect}) over the orbit category $\sO_G$, and as a result, the generators added in the $n$-th stage can have degrees greater than $n$.


 The construction of an elementary extension of a system of $\DGA$s $\cU$ over $\sO_G$, depends on the following data; a diagram of vector spaces $\underbar{V}$ over $\sO_G$ of degree $n$ and an element $[\alpha]\in H^n(\cU;\underbar{V})$, and the extension is denoted by $\cU^{\alpha}(\underbar{V})$. 
  Any two non-isomorphic minimal systems of $\DGA$s, with an isomorphic diagram of cohomology algebras, differ at some $n$-th stage.

Moreover, we observe that at each stage of the construction of a minimal system, if the associated diagram of vector spaces (\Cref{def:associ:vect}) involved in the elementary extension is injective, then 
the computations simplify. We therefore, consider the following question:

\begin{question}\label{question: 2}
   What conditions on an injective cohomology diagram of graded algebras ensure that the associated diagram of vector spaces (\Cref{def:associ:vect}) added for each elementary extension is injective?
    
\end{question}

In \cite{santhanam2023equivariant} the authors study \Cref{question: 2} for the case when the group $G$ is the cyclic group $C_p$, where $p$ is a prime. They show that if a cohomology diagram $\cA^\ast$ over the orbit category $\sO_{C_p}$ has structure maps that admit a retraction (see \Cref{defn: retract}), and if $\sM$ is the minimal model of $\cA^\ast$, then the associated diagram of vector spaces at each step of the construction of $\sM$ is injective. As a consequence, the equivariant minimal model $\sM$ is \emph{level-wise minimal}, i.e., for each $H \leq G$, the component $\sM(G/H)$ is a minimal model of $\cA^\ast(G/H)$, thereby reducing the problem to the non-equivariant case (see \Cref{Z_P injimpliesminimal}).

In this article, we prove a \emph{combination-type theorem} for the case $G = C_{pq}$, the cyclic group of order $pq$ with $p$ and $q$ distinct primes. We begin with two cohomology diagrams, $\cA_1$ over $\sO_{C_p}$ and $\cA_2$ over $\sO_{C_q}$, each having structure maps that are retracts. We consider their equivariant wedge product (see \Cref{defn:wedgeequivar}), which gives rise to a diagram over $\sO_{C_{pq}}$. We show that this resulting diagram of $\DGA$s is injective (see \Cref{prop: combination injective}) and that the associated diagram of vector spaces added at each stage of the minimal model construction remains injective. In particular, the minimal model of this diagram is level-wise minimal (see \Cref{thm:combmain}). This provides a structured method for constructing new equivariant minimal models from existing ones that preserve desirable properties. 
We also prove \Cref{cor:main} that identifies a class of $C_{pq}$-spaces whose equivariant minimal models are level-wise minimal. That is, at each orbit level, the model coincides with the non-equivariant minimal model. These examples show how, in the absence of nontrivial injective resolutions, the equivariant structure simplifies significantly.

Recall that a cohomologically $1$-connected $\DGA$ is said to be \emph{formal} if its minimal model is weakly equivalent to the minimal model of its cohomology algebra. A simply connected topological space is called \emph{formal} if its minimal Sullivan model is formal. For such spaces, the rational homotopy groups can be computed directly from the cohomology algebra using the rational Postnikov tower~\cite[Theorem 3.3]{DGMS}. Many authors have extensively studied formality in various contexts, including \cite{DGMS, HARST, FELIXHALPFORMAL, FormalMARTIN, LUPOPERAFORMAl} and references therein.

A system of $\DGA$s over the orbit category of a group $G$ is said to be \emph{equivariantly formal} if its minimal model is weakly equivalent (see \Cref{def:weakyeqivalent}) to the injective envelope of its cohomology diagram (see~\cite{FINE, santhanam2023equivariant} for examples). A $G$-space is said to be equivariantly formal if the associated minimal system of $\DGA$s is equivariantly formal.

With this setup, we pose the following questions:

\begin{question}\label{qstn:3}
Let $X$ and $Y$ be formal $C_p$- and $C_q$-spaces, respectively, where $p$ and $q$ are distinct primes. Is the wedge $X \vee Y$ also formal as a $C_{pq}$-space? (See \Cref{defn: equiwedgespace}.)
\end{question}

We conclude this article by addressing \Cref{qstn:3}. In \Cref{thm:main2}, we provide a positive answer, establishing a broad class of $C_{pq}$-formal spaces. We conclude with a concrete example illustrating this result.

\textbf{Overview.} In \Cref{secion: background}, we provide necessary definitions and review known results in both the equivariant and non-equivariant settings. In \Cref{section:injecti}, we introduce \emph{Property~I} and prove that a diagram of $\DGA$s over $\sO_{C_{pq}}$ satisfying this property is injective (see \Cref{thm: injective C_pq}). We also give examples and show that restricting a $C_{pq}$-action on a space to a subgroup of order $p$ or $q$ yields injective cohomology diagrams in $Vec^*_{C_p}$ and $Vec^*_{C_q}$, respectively. In \Cref{section: combnation theorem}, we establish the combination theorem for minimal models (see \Cref{thm:combmain}) as outlined in \Cref{section:intro}. Finally, we present a result that provides a class of $C_{pq}$-formal spaces (see \Cref{exmp: mainnnn}), illustrating how formality behaves under such equivariant constructions (\Cref{thm:main2}).

%

    \section{Background}\label{secion: background}
We restrict ourselves to finite groups and $\DGA$s over $\QQ$ throughout. We start with some definitions.

\begin{definition}\label{defn:inj object}(Injective object in a category)
    
An object $I$ in a category $\mathcal{C}$ is said to be \textit{injective} if for every injective morphism $f:X\to Y$ and every morphism $g:X\to I$

 \begin{figure}[h]
     \centerline{
\xymatrix{
X \ar[d]_g\ar[r]^f& Y\ar@{-->}[dl]^{h} \\
I } }
\caption{$f$ is injective, $g$ is any morphism}
\label{fig:injec}
\end{figure}

there exists a morphism $h:Y\to I$ such that $h\circ f=g$ (see \Cref{fig:injec}). 
\end{definition}

\begin{definition}\label{defn:orbit}(Category of canonical orbits)
Given a group $G$, the category of canonical orbits is the category whose objects are $G$-sets $G/H$ and morphisms are $G$ maps between them. We denote this category by $\sO_G$. 
\end{definition}

A \emph{diagram of $\DGA$s} is a covariant functor from the orbit category $\sO_G$ to the category of cohomologically $1$-connected differential graded algebras. The category of such diagrams is denoted by $\sO_G[\DGA]$. A \emph{dual rational coefficient system} is a covariant functor from $\sO_G$ to the category of graded rational vector spaces, and the corresponding category is denoted by $Vec^\ast_G$. In contrast, a \emph{rational coefficient system} is a \emph{contravariant} functor from $\sO_G$ to rational vector spaces, with category denoted by $Vec_G$.

A diagram of $\DGA$s is called a \emph{system of $\DGA$s} if, when regarded as an object in $Vec^\ast_G$ (i.e., by forgetting the differential and multiplication), it is injective in that category. This terminology follows \cite{GT82}. The category of such injective diagrams is denoted by $\DGA^{\sO_G}$.


Let $X$ be a $G$-space such that for every subgroup $H \leq G$, the fixed point space $X^H$ is nonempty and simply connected. Then the cohomology diagram of $X$ (with zero differential) is $1$-connected, i.e., $H^1(X^H; \mathbb{Q}) = 0$ for all $H \leq G$. However, such a diagram need not be injective as a dual coefficient system. Nevertheless, every dual rational coefficient system admits an injective envelope.

We now describe \cite[Prop. 7.34]{LS02}, the  embedding of a given coefficient system  $\u M$ into its injective envelope $\cI$.

\begin{definition}\label{equation:24}

We define 
\beq 
V_H:=\bigcap_{H\subset K}\ker\u M(\hat{e}_{H,K}),
\eeq

 where $\hat{e}_{H,K}:G/H\to G/K$ is the projection and $M(\hat{e}_{H,K})$ is the induced structure map on the functor $M$.  Note that  $V_G$ is defined to be $M(G/G)$. Let $\cI=\oplus_H \underline{V}_H$, where 
 \beq
 \underline{V}_H(G/K):=\hom_{\QQ(N(H)/H)}(\QQ(G/H)^K,V_H).
 \eeq

There is an injective morphism $M\to \cI$ extending the natural inclusions of $\bigcap_ {H\subset K}\ker\u M(\hat{e}_{H,K})$.

\end{definition}

\begin{prop}\cite[Section 4]{GT82}\label{sum inj}
    A dual coefficient system $M$ is injective if and only if it is of the form $M = \bigoplus_H \underline{V}_H$ for some collection of $\QQ(N(H)/H)$-modules
$V_H$ and $$\underline{V}_H(G/K)=\hom_{\QQ(N(H)/H)}(\QQ(G/H)^K,V_H).$$
\end{prop}

Given a diagram of $\DGA$s, forgetting the differential will give a dual rational coefficient system whose injective envelope is a diagram of $\DGA$s, with $0$ differential. However, the map into the injective envelope of dual rational coefficient system will not be a quasi-isomorphism in general. %

Fine and Triantafillou \cite{FINE}, prove the existence of \emph{injective envelope} for a diagram of $\DGA$s.

\begin{prop}\cite[Theorem 1]{FINE}\label{thm:fineee}
For a diagram of $\DGA$s $\sA$ over $\sO_G$, where $G$ is finite group, there is an injective system  of $\DGA$s  $\sI(\sA)$, called the injective envelope of $\sA$ along with an inclusion $i:\sA\to \sI(\sA)$ which is a quasi-isomorphism. 
\end{prop} 

In \cite{santhanam2023equivariant}, the authors establish a necessary and sufficient condition for the injectivity of a $C_p$-diagram of $\DGA$s. Their result characterizes precisely when such a diagram defines an injective system over the orbit category $\sO_{C_p}$.

\begin{prop}\cite[Proposition 4.1]{santhanam2023equivariant}\label{prop:inj C_p}
Let $G=C_p$, where $p$ is a prime. Let $\cA\in \DGA^{\sO_G}$. Then  $\cA$ as an element of $Vec^{\ast}_G$ is injective if and only if the map $\cA(\hat{e}_{e,G}): \cA(G/e)\to \cA(G/G)$ is surjective. 
\end{prop}

\begin{exmp}\label{exmp:noninjective}
Consider the $G$-space $X=S^3$, where $G=C_2$ acts on $S^3$ by reflection, which fixes the equator sphere $S^2$. So here $G=C_2$, $X^G=S^2$ and $X^e=S^3$. The corresponding cohomology diagram is given by $H^\ast(X;\underline{\QQ})$, which is not injective. This follows from Proposition \ref{prop:inj C_p}.  
\end{exmp}

We now introduce the notion of weak equivalences in the category $\sO_G[\DGA]$. To do so, we first define homotopy between morphisms of systems of $\DGA$s.

\begin{definition}
Let $\sU$ be a system of $\DGA$s. Define the tensor extension $\sU(t,dt)$ as the diagram of $\DGA$s given by
\[
\sU(t,dt)(G/H) := \sU(G/H) \otimes_\QQ \QQ(t,dt),
\]
where $\QQ(t,dt)$ denotes the free commutative $\DGA$ on a generator $t$ in degree $0$ and $dt$ in degree $1$. 

Two morphisms of $\DGA$ diagrams $f,g: \sU_1 \to \sU_2$ are said to be \emph{homotopic} (written as $f\simeq g$) if there exists a $\DGA$ morphism
\[
H: \sU_1 \to \sU_2 \otimes \QQ(t,dt)
\]
such that the evaluations $p_0(H) = f$ and $p_1(H) = g$, where $p_i: \QQ(t,dt) \to \QQ$ is defined by $p_i(t) = i$ and $p_i(dt) = 0$ for $i = 0,1$.
\end{definition}

This notion of homotopy does not define an equivalence relation in general on the category $\DGA^{\sO_G}$. However, it becomes an equivalence relation when restricted to \emph{minimal} systems of $\DGA$s (see \Cref{def:minimal}), a fact we elaborate on later in this section.

To define a broader equivalence relation on diagrams of $\DGA$s, we introduce quasi-isomorphisms. Given two diagrams $\sU, \sB \in \sO_G[\DGA]$, a morphism $f: \sU \to \sB$ (or $f: \sB \to \sU$) is said to be a \emph{quasi-isomorphism} if it induces an isomorphism in cohomology at each level; that is, $H^\ast(f_{G/H}): H^\ast(\sU(G/H)) \to H^\ast(\sB(G/H))$ is an isomorphism for all subgroups $H \leq G$. 

The equivalence relation generated by quasi-isomorphisms is referred to as a \emph{weak equivalence} between diagrams of $\DGA$s. Using the notion of injective envelopes, we extend this to arbitrary diagrams as follows:

\begin{definition}\label{def:weakyeqivalent}(Weak equivalence)
Let $\cU$ and $\cV$ be diagrams of $\DGA$s over $\sO_G$. We say that $\cU$ and $\cV$ are \emph{weakly equivalent} if their injective envelopes are weakly equivalent as systems of $\DGA$s.
\end{definition}

Recall that associated with any $G$-space $X$, there is the system of $\DGA$s given by the de Rham-Alexander-Spanier algebra $\sE(X)(G/H):=\cA(X^H)$ for every $H\leq G$. Triantafillou \cite[Theorem 1.5]{GT82} proves that there is a bijective correspondence between $G$-space $X$ (with every fixed point set simply connected) and the minimal system of $\DGA$s $\sM_X$ of $\sE(X).$

Scull generalizes these ideas to spaces with an $S^1$ action. In \cite[Section 21]{LS02}, Scull shows that, unlike the non-equivariant case, the notion of minimality in the equivariant case arising from filtration via minimal extensions of systems of $\DGA$s does not satisfy the decomposability condition. 

Note that homotopy defines an equivalence relation on morphisms from $\sM \to \sB$ for any system of $\DGA$ $\sB$, whenever $\sM$ is a minimal system \cite[Prop. 3.5]{LS02}. Further, given a quasi isomorphism $\rho: \sU \to \sB$ of a system of $\DGA$s and a morphism $f:\sM\to \sB$ is any map from a minimal system $\sM$, there is a lift $g:\sM\to \sU$ such that $\rho g\simeq f$, \cite[Prop. 3.6]{LS02}.

Their results (\cite[Theorem 1.5]{GT82}, \cite[Theorem 4.13]{LS02}) show that the category of $G$-spaces (whose fixed points sets are simply connected)  up to rational homotopy equivalences is equivalent to the category of minimal systems of $1$-connected $\DGA$s modulo homotopy equivalences.

In order to give the construction of a minimal model of a system of $\DGA$s we first define elementary extensions.  

\begin{definition}\label{def:ele}(Elementary extension)
Given a system of $\DGA$s $\mathcal{U}$, a diagram of vector spaces $\underline{V}$ assigned to be of degree $n$, and a map $\alpha: \underline{V}\to \underline{Z}^{n+1}(\mathcal{U})$ (here $\underline{
Z}(\cU)$ denotes the kernel of $\cU$), the {\it elementary extension } of $\mathcal{U}$ with respect to $\alpha$ and $\underbar{V}$, denoted by $\mathcal{U}^{\alpha}(\underbar{V})$, is constructed as follows.

Let $\underbar{V}\to \underbar{V}_0 \xrightarrow{w_0} \underbar{V}_1\xrightarrow{w_1} \underbar{V}_2\cdots$
be minimal injective resolution of $\underbar{V}$ constructed by taking $\underbar{V}_i$ to be the injective embedding of $\operatorname{coker} w_{i-1}$, which is of finite length.  

Construct a commutative diagram (see \Cref{fig:elementaryext}).

\begin{figure}[h]
\centerline{
\xymatrix{
\underline{V}\ar[d]_{\alpha} \ar[r] & \underline{V}_0\ar[d]^{ \alpha_0}\ar[r]^{w_0}& \underline{V}_1\ar[d]^{\alpha_1}\ar[r]^{w_1} & \underline{V}_2\ar[r]\ar[d]^{\alpha_2} & \cdots  \\
\underline{Z}^{n+1}(\mathcal{U})\ar[r]         & \mathcal{U}^{n+1}\ar[r]_{d}&\mathcal{U}^{n+2}\ar[r]_d & \mathcal{U}^{n+3}_d \ar[r] &\cdots
}}
\caption{Diagram for elementary extension}
\label{fig:elementaryext}
\end{figure}

The  maps $\alpha_i$ are  constructed inductively by  first noting that $d\alpha_iw_{i-1}=dd\alpha_{i-1}=0$, so $d\alpha_i|_{\Im w_{i-1}}=0$ and then by the injectivity of $\cU$ we get a commutative diagram (see \Cref{fig:UU}).

\begin{figure}[h]
\centerline{
 \xymatrix{
\underbar{V}_i/\Im w_{i-1} \ar[d]_{d\alpha_i}\ar@{^{(}->}[r]^(0.6){ \rho^{\ast}} &\underbar{V}_{i+1}\ar[dl]^{\alpha_{i+1}}\\
\cU^{n+i+1}}}

\caption{$\alpha_{i+1}$ is constructed using injectivity of $\cU$}
\label{fig:UU}
\end{figure}

Define $\cU^{\alpha}(\underbar{V}):=\cU\otimes(\otimes_i\QQ(\underbar{V}_i))$, where $\QQ(\underbar{V}_i)$ is the free graded commutative algebra generated at $G/H$ by the vector space $\underbar{V}_i(G/H)$ in degree $n+i$; the differential is defined on $\cU$ by the original differential on $\cU$, and on the generators of $\underbar{V}_i$ by $d=(-1)^i\alpha_i+w_i$. Since  $\underbar{V}_i$ is injective for all $i$ by construction, as a vector space the system is the tensor product of injectives and hence injective. Thus, $\cU^{\alpha}(\underbar{V})$ is a new system of $\DGA$s.

\end{definition}

  A minimal system of $\DGA$s is  defined as follows. 
\begin{definition}\label{def:minimal}(Minimal system of $\DGA$s)
A system of $\DGA$s $\sM$ is minimal if $\sM=\cup_n\sM_{n}$, where $\sM_{0}=\sM_{1}=\underline{\mathbb{Q}}$ and $\sM_{n}=\sM_{n-1}(\underline{V})$ is the elementary extension for some diagram of vector spaces $\underline{V}$ of degree $\geq n$. 
\end{definition}

\begin{theorem}\label{thm:quasiisiso}[\cite{LS02} Theorem 3.8]
If $f:\sM\to \sN$ be a quasi-isomorphism between minimal systems of $\DGA$s, then $f\simeq g$, where $g$ is an isomorphism. 

\end{theorem}

Thus, if we have two minimal systems $\sM$, $\sN$ and quasi-isomorphisms $\rho_1:\sM\to \cU$ and $\rho_2:\sN\to \cU$ by the lifting property of maps from minimal systems to systems of $\DGA$s we get a map $f:\sM\to \sN$ which is a quasi-isomorphism. By Theorem \ref{thm:quasiisiso}, we get $f\simeq g$ where $g$ is an isomorphism. Now we define the following.

\begin{definition}\label{def:minimalmdel}(Minimal model)
If $\sM$ is a minimal system and $\rho: \sM\to \sU$ is a quasi-isomorphism, we say that $\sM$ is a minimal model of $\sU$. 
\end{definition}

Maps between two minimal systems of $\DGA$s are much {\it nicer}, in the sense that they are always homotopy to a level-wise map of extensions. We will make use of this fact later. 
\begin{lemma}\cite [Lemma 13.57]{LS02}\label{lemm:13.57}
Any morphism $f:\sM\to \sN$ between minimal systems of $\DGA$ is homotopic to a morphism $g$ which maps $\sM_{n}$ to $\sN_n$ for all $n$.

\end{lemma}

Observe that any minimal system is cohomologically $1$-connected, that is, it satisfies $\underline{H}^0(\sM)=\mathbb{Q}$ and $\underline{H}^1(\sM)=0$. It can be shown that being cohomologically $1$-connected is sufficient for a diagram of $\DGA$s to have a {\it minimal model.}


\begin{theorem}\cite [Theorem 3.11]{LS02}\label{thm:existance of minimal}
If $\sU$ is a system of $\DGA$s which is cohomologically $1-$connected, then there exists a minimal model of $\sU$, i.e., a minimal system $\sM$ and a quasi-isomorphism $\rho:\sM \to \sU$.
\end{theorem}

\begin{definition}\label{def:associ:vect}(Associated diagram of vector spaces)
Let $\cA$ be an injective cohomology diagram and $(\sM,\rho)$ be its minimal model. Let $\sM_n=\sM_{n-1}^{\gamma}(\underbar{V})$, be the $n$-th stage construction of $\sM$, which is obtained by taking elementary extension of $\sM_{n-1}$ with the injective resolution of $\underbar{V}$. We refer to $\underbar{V}$ as the  \emph{$n$-th stage associated diagram of vector spaces of $\cA$}. 
\end{definition}

In the non-equivariant setup, we define the following.
\begin{definition}[Retract]\label{defn: retract}
   Given a $\DGA$  $A$ and a sub-$\DGA$ $B$ of $A$, we say $B$ is a retract of $A$ if there is a $\DGA$-morphism $r: A\to B$ such that $r\circ i=id$. Here $i:B\to A$ is the inclusion morphism and the morphism $r$ is called the retraction. Equivalently, we define a map $r:A \to B$ is said to be a retraction if there is a $\DGA$ morphism $i:B\to A$ such that $r\circ i=id$. 
\end{definition}

In the corrected version of \cite{santhanam2023equivariant} the authors prove the following result. We provide the prove for the sake of completeness.
\begin{prop}\label{Z_P injimpliesminimal}
Let $G=C_p$, for $p$ prime. If the structure map in the cohomology diagram $\cA$  $\cA_{e,G}: \cA(G/e)\to \cA(G/G)$ is a retraction of $\DGA$s then the associated diagram of vector spaces is injective. In particular, the minimal model of the cohomology diagram is level-wise minimal.
\end{prop}
\begin{proof}

Since the structure map $\cA_{e,G}:\cA(G/e)\to \cA(G/G)$ is a retraction, there exists $i:\cA(G/G)\to \cA(G/e)$ such that $\cA_{e,G}\circ i=id$. This  implies $\cA_{e,G}$ is surjective, and it follows that $\cA$ is injective diagram  of graded algebras. Note that for any minimal system of $\DGA$s $\sN$, the $\DGA$ $\sN(G/G)$ is non-equivariantly minimal by construction. 
    Let $\rho:\sM\to \cA$ be the minimal model and let $\sM_{e,G}:\sM(G/e)\to \sM(G/G)$ be the corresponding structure map. We claim that there exists $j:\sM(G/G)\to \sM(G/e)$ an inclusion map of $\DGA$s. 
    Since $\cA(G/e)$ is a $\DGA$ with zero differential we have, $\rho(G/e):\sM(G/e)\to \cA(G/e)$ is surjective quasi-isomorphism, by Lifting Lemma (\cite[Lemma 12.4]{FHT}), there exists a lift $j:\sM(G/G)\to \sM(G/e)$ so that the \Cref{fig:levelwisemin} commutes  
        \begin{figure}[h]
        \centerline{
\xymatrix{
\sM(G/e)\ar[r]^{ \rho(G/e)}  & \cA^i(G/e) \\
\sM(G/G) \ar[r]_{\rho(G/G)}\ar[u]^{j}   & \cA^i(G/G)\ar[u]_{i}
}}
\caption{Horizontal arrows are quasi-isomorphisms, $i,j$ are inclusions }
\label{fig:levelwisemin}
\end{figure}
Therefore, 
\begin{equation}
    \cA_{e,G}\circ i \circ \rho(G/G)= \cA_{e,G} \circ \rho(G/e)\circ j\\
    \implies \rho(G/G)=\rho(G/G)\circ \sM_{e,G}\circ j
\end{equation}
Since $\rho(G/G)$ is a quasi-isomorphism , $\sM_{e,G}\circ j:\sM(G/G)\to \sM(G/G)$ is a quasi-isomorphism. It then follows that, $\sM_{e,G}\circ j$ is an isomorphism and therefore $j:\sM(G/G)\to \sM(G/e)$ is an inclusion and $\sM_{e,G}$ is a surjection. Also note that if we let $\sM_{e,G}\circ j=\phi$, where $\phi$ is isomorphism, then $\sM_{e,G}\circ j\circ \phi^{-1}=id$, making $\sM_{e,G}$ a retraction.


Next, we show that all the associated systems of vector spaces are injective by induction on $n$ where $\sM=\cup \sM_m$. 
    Recall $\sM_n=\sM_{n-1}(\underline{V})$, where $\underline{V}$ is $H^{n+1}(\ker (\beta) \oplus \underline{\QQ})$ is the associated diagram of vector spaces at $n$-th stage. Any element of $\underbar{V}(G/G)$ looks like the product of the elements of $\sM_{n-1}(G/G)$, $\cA(G/G)$ and $\sum \cA(G/G)$. We study case by case to conclude that $\underbar{V}(\hat{e}_{e,G})$ is surjective. Let $[x]\in \underbar{V}(G/G)$.

    \begin{enumerate}
        \item 
        If $x\in \sM^{n+1}_{n-1}(G/G)$ i.e., $[x]\in \underline{V}(G/G)$. Then $j(x)\in \sM_{n-1}(G/e)$. 

        As $\beta=\rho$ on $\sM_{n-1}$, we have $i\circ \rho(G/G)(x)=\rho(G/e)\circ j(x)$, this implies $i\circ \beta(G/G)(x)=\beta(G/e)\circ j(x)$, which implies $j(x)\in \ker \beta(G/e)$. 
        As $j$ is a $\DGA$-map we get $dj(x)=jd(x)$, which gives $j(x)\in \underline{V}(G/e)$. 
        \item If $x\in \sum \cA(G/G)$ then by injectivity of $\cA$, one gets a pre-image in $\sum \cA(G/e)$. As the differential is zero for elements in $\sum \cA$. So we get a pre-image in $\underbar{V}(G/e)$. 

        \item Assume $x$ is the product of elements in $\sM_{n-1}$, $\cA$ and $\sum \cA$. In this case, note that the maps $i,j$ induces a $\DGA$-map $g: \sM(G/G)\otimes\QQ (\cA \oplus \sum \cA)(G/G)\to \sM(G/e)\otimes\QQ (\cA \oplus \sum \cA)(G/e)$. If $x=m.a.sb$ where $m\in \sM(G/G)$, $a\in \cA(G/G)$ and $sb\in \sum \cA(G/G)$, with $[x]\in \underline{V}(G/G)$, then we have $[g(m.a.sb)]\in \underline{V}(G/e)$.

    \end{enumerate}

Hence, $\underbar{V}$ is injective and it follows that  the minimal model $\sM$ is level-wise minimal. 
\end{proof}

Surjective morphisms between non-equivariant minimal models exhibit useful structural properties. One such consequence is the following result.

\begin{lemma}\label{lemma: minimal surjective}
Let $\varphi \colon A \to B$ be a morphism of $\DGA$s, where both $A$ and $B$ are minimal Sullivan algebras (i.e., non-equivariantly minimal). 

For each $k \geq 1$, let $A_{\leq k}$ and $B_{\leq k}$ denote the subalgebras of $A$ and $B$ generated by elements of degree less than or equal to $k$. Then, by the structure of minimal Sullivan algebras (via Hirsch extensions), we can write:
\[
A_{\leq k} = A_{\leq k-1} \otimes \wedge V^k 
\quad \text{and} \quad
B_{\leq k} = B_{\leq k-1} \otimes \wedge W^k,
\]
where $V^k$ and $W^k$ are $\mathbb{Q}$-vector spaces concentrated in degree $k$.

If the map $\varphi \colon A \to B$ is surjective, then for each $k$, the induced map $A_{\leq k} \to B_{\leq k}$ is also surjective. In particular, the induced linear map $V^k \to W^k$ is surjective for each $k$.
\end{lemma}

\begin{proof}
Consider the projection $\pi_W : B \to W$ be the projection onto generators, using the decomposition as a direct sum $$B = W \oplus W^{\otimes 2}/S_2 \oplus \dots.$$
Let $\phi|_V$ be the restriction of $\phi$ to generators. Then $\pi_W \circ \phi|_V : V \to W$ is the linear part of $\phi$.

Since $B_{\leq k}$ is generated by elements of degree $\leq k$, it suffices to prove that $A_{\leq k} \to B_{\leq k}$ hits all of these elements. But if $y \in B_{\leq k}$ is of degree $\leq k$, since $\phi$ is surjective, there exists $x \in A$ with $\phi(x) = y$; and $\deg(x) = \deg(y) \leq k$ so $x \in A_{\leq k}$.

Moreover, the linear part of $f\phi$ is also surjective in every degree. Suppose $w \in W^k$ is a generator of degree $k \geq 2$. Since $\phi$ is surjective, there exists $a \in A$ such that $f(a) = w$. The element $a$ decomposes as $a = v + a'$, where $v \in V$ and $a' = b_1 c_1 + \dots + b_l c_l \in A$ is decomposable as a sum of nontrivial products of elements of positive degree. It follows that $w = f(a) = f(v) + \sum_i f(b_i) f(c_i)$ projects to $f(w)$ under $\pi_W$. But of course $\pi_W(w) = w$ so $\pi_W(f(v)) = w$.
\end{proof}

\begin{definition}[Pullback] \label{defn:pullback}
Let $A, B, C$ be $\DGA$s over $\QQ$, and let $f: A \to C$ and $g: B \to C$ be morphisms of \DGA s. The \emph{pullback DGA}, $A \times_C B$ is defined as follows:

\begin{itemize}
  \item As a graded module,
  \[
  (A \times_C B)_n = \{ (a, b) \in A_n \times B_n \mid f(a) = g(b) \in C_n \}.
  \]

  \item The multiplication is defined componentwise:
  \[
  (a_1, b_1) \cdot (a_2, b_2) = (a_1 a_2, b_1 b_2).
  \]

  \item The differential is defined by:
  \[
  d(a, b) = (d_A a, d_B b).
  \]
\end{itemize}

This makes $A \times_C B$ into a $\DGA$, and it satisfies the universal property of the pullback in the category of $\DGA$s i.e., given $h:D\to A$ and $e:D\to C$ so that $fh=ge$ there exists $\tau:D\to A\times_C B$ so that the following diagram commutes (\Cref{Fig:4}). 

\begin{figure}[h] \centerline{%
 \xymatrix{
D\ar@{-->}[dr]^{\tau}\ar@/_2.0pc/@{->}@[black][ddr]_{h}\ar@/^2.0pc/@{->}@[black][rrd]^{e} & & \\
& A \times_C B\ar[r]^{\pi_2}\ar[d]_{\pi_1} & B\ar[d]^{g} \\
& A\ar[r]_{f} & C 
 }}
\caption{Pull-back diagram}\label{Fig:4}
\end{figure}
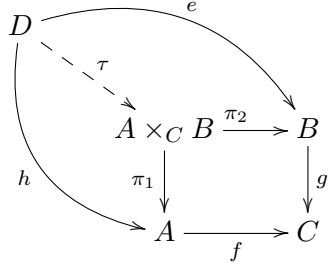

The map $\tau $ is given by $\tau(d)=(h(d),e(d))$ for any $d\in D$.
\end{definition}

\begin{remark}
Consider the diagram in \Cref{defn:pullback}. If $f,g$ are surjective then $\pi_i$ for $i=1,2$ are surjection. Additionally, if $\tau$ is surjection so are $h,e$. This can be seen in the following way.  
Let $a\in A$. As $g$ is surjection there exists $b\in B$ so that $f(a)=g(b)$ so that $(a,b)\in A\times _C B$. Hence $\pi_1$ is surjective. Using a similar argument, we conclude that $g$ is surjective. 
Additionally, if we assume $\tau$ is surjective it follows that $h,e$ are surjective.   
\end{remark}

Motivated from \cite[Example 2.47]{FOT2008} we define the following.

\begin{definition}(Wedge product of $\DGA$s)\label{defn:wedge of DGAs}
    Let \( A = (A^\ast, d_A) \) and \( B = (B^\ast, d_B) \) be $\DGA$s over \( \mathbb{Q} \), each with unit \(1_A\) and \(1_B\). The \emph{wedge sum} \( A \vee B \) is defined as the pushout in the category of DGAs over \( \mathbb{Q} \), i.e.,
\[
A \vee B := (A \oplus B) / \langle 1_A - 1_B \rangle,
\]
with the following structure:
\begin{itemize}
    \item The unit is \(1 := 1_A = 1_B\).
    \item The multiplication is given by:
    \[
    (a + b) \cdot (a' + b') = aa' + bb',
    \]
    with the additional condition that
    \[
    a \cdot b' = 0 = b \cdot a' \quad \text{for all } a, a' \in A, \; b, b' \in B.
    \]
    \item The differential is defined componentwise:
    \[
    d(a + b) = d_A(a) + d_B(b).
    \]
\end{itemize}

This DGA reflects the cohomology ring of a wedge sum of spaces, where cohomology classes from the two summands do not interact except through the common unit.
\end{definition}

Let $f: A \to A'$ and $g: B \to B'$ be morphisms of $\DGA$s, and assume that $A \vee B$ and $A' \vee B'$ are defined as in \Cref{defn:wedge of DGAs}.

Then the induced map
\[
f \vee g : A \vee B \longrightarrow A' \vee B'
\]
is defined on representatives by
\[
(f \vee g)([a + b]) := [f(a) + g(b)],
\]
for all $a \in A$, $b \in B$. This is well-defined because $f(1_A) = 1_{A'}$, $g(1_B) = 1_{B'}$, and $f(a)g(b) = 0$ in $A' \vee B'$ whenever $a \in \overline{A}$, $b \in \overline{B}$.

Moreover, since $f$ and $g$ are DGA maps, $f \vee g$ respects the differential:
\begin{align*}
d(f \vee g)([a + b]) 
&= d([f(a) + g(b)]) \\
&= [d(f(a)) + d(g(b))] \\
&= [f(da) + g(db)] \\
&= (f \vee g)([da + db]) \\
&= (f \vee g)(d[a + b]).
\end{align*}

\begin{lemma}
    Let $f:A\to A'$ and $g:B\to B'$ be retraction of $\DGA$s then the induced map $f\vee g:A\vee A'\to B\vee B'$ is also a retraction.
\end{lemma}
\begin{proof}
As $f,g$ are retractions there exist $i:A'\to A$ and $j:B'\to B$ such that $f\circ i=id_{A'}$ and $g\circ j=id_{B'}$.

As above we can define $i \vee j : A' \vee B' \to A \vee B$ by
\[
(i \vee j)([a' + b']) := [i(a') + j(b')]
\]
for $a \in A'$, $b \in B'$. 

Then we compute:
\[
(f \vee g) \circ (i \vee j)([a + b]) = (f \vee g)([i(a) + j(b)]) = [f(i(a)) + g(j(b))] = [a + b],
\]
using that $f \circ i = {id}_{A'}$ and $g \circ j = {id}_{B'}$.

Thus, $(f \vee g) \circ (i \vee j) = {id}_{A \vee B}$, and $f \vee g$ is a retraction.
\end{proof}

\begin{lemma}\label{lemma: zigzag weak}
Let $A \simeq A'$ and $B \simeq B'$ be DGAs, where the weak equivalences are given by zigzags of quasi-isomorphisms. Then the wedge (coproduct) $A \vee B$ is weakly equivalent to $A' \vee B'$, i.e.,
$$
A \vee B \simeq A' \vee B'.
$$
\end{lemma}

\begin{proof}
Suppose we have zigzags of quasi-isomorphisms connecting $A$ to $A'$ and $B$ to $B'$ as follows:
$$
A = A_0 \xleftrightarrow{\simeq} A_1 \xleftrightarrow{\simeq} \cdots \xleftrightarrow{\simeq} A_n = A', \quad
B = B_0 \xleftrightarrow{\simeq} B_1 \xleftrightarrow{\simeq} \cdots \xleftrightarrow{\simeq} B_m = B'.
$$
Assume without loss of generality that $n \leq m$. We can extend the shorter zigzag (say, the $A_i$'s) by repeating the terminal DGA $A_n$ as needed, defining $A_i = A_n$ for $i > n$, so that both sequences have the same length.

Then we define a zigzag of quasi-isomorphisms between the wedge products:
$$
A_0 \vee B_0 \xleftrightarrow{\simeq} A_1 \vee B_1 \xleftrightarrow{\simeq} \cdots \xleftrightarrow{\simeq} A_m \vee B_m,
$$
where each map is the wedge (coproduct) of quasi-isomorphisms between the respective terms. Since the wedge of quasi-isomorphisms between cofibrant DGAs is again a quasi-isomorphism, this sequence forms a zigzag of quasi-isomorphisms from $A \vee B$ to $A' \vee B'$.

Hence, $A \vee B \simeq A' \vee B'$.
\end{proof}

    \section{Injectivity results}\label{section:injecti}

    In this section, we consider $G=C_{pq}$, be the cyclic group of order $pq$, where $p$ and $q$ are distinct prime numbers. For a $G$-space $X$ and a subgroup $H$ of $G$, the fixed point space for $H$ is defined as $X^H:=\{x\in X~|~h.x=x~\text{for all }h\in H\}$. Now $G$ has four subgroups of order $1,p,q,pq$ and which we denote by $e,P,Q,G$ respectively. Let $\cA\in \sO_{C_{pq}}[\DGA]$ (see \Cref{Fig:222}).

\begin{figure}[h] \centerline{%
 \xymatrix{\cA(G/e)\ar[d]_{\cA_{e,P}}\ar[r]^{\cA_{e,Q}} & \cA(G/Q)\ar[d]^{\cA_{Q,G}}\\
    \cA(G/P)\ar[r]_{\cA_{P,G}} & \cA(G/G)
 }}
\caption{$\cA$ as $\sO_{C_{pq}}$-diagram}\label{Fig:222}
\end{figure}

    where the vertical and horizontal maps are the structure maps.
Consider the maps $$\cA(G/P)\xrightarrow{\cA_{P,G}} \cA(G/G)\xleftarrow{\cA_{Q,G}}\cA(Q,G)$$ and take the pullback $\DGA$ (as in \Cref{defn:pullback}) call it $\cK$ so that we have the following commutative square (\Cref{Fig:3}).

\begin{figure}[h] \centerline{%
 \xymatrix{
  \cK\ar[d]_{\pi_1}\ar[r]^{\pi_2} & \cA(G/Q)\ar[d]^{\cA_{Q,G}}\\
    \cA(G/P)\ar[r]_{\cA_{P,G}} & \cA(G/G)
    }}
\caption{$\cK$ is the pull-back with respect to $\cA_{P,G}$ and $\cA_{Q,G}$}\label{Fig:3}
\end{figure}

    By the universal property of pullback there is a unique map $\cA_{\cK}: \cA(G/e) \to \cK$ so that the following diagram commutes (see \Cref{Fig:33}).

\begin{figure}[h] \centerline{%
 \xymatrix{
\cA(G/e)\ar@{-->}[dr]^{\cA_{\cK}}\ar@/_2.0pc/@{->}@[black][ddr]_{\cA_{e,P}}\ar@/^2.0pc/@{->}@[black][rrd]^{\cA_{e,Q}} & & \\
& \cK\ar[r]^{\pi_2}\ar[d]_{\pi_1} & \cA(G/Q)\ar[d]^{\cA_{Q,G}} \\
& \cA(G/P)\ar[r]_{\cA_{P,G}} & \cA(G/G) 
 }}
\caption{The map $\cA_{\cK}$ exists due to universal property of pull-bakc}\label{Fig:33}
\end{figure}

With these notations in mind, we define the following.
\begin{definition}\label{defn: property I}(Property I)
    Let $\cA\in Vec^*_G$ where $G=C_{pq}$, $p,q$ distinct primes.  We say that $\cA$ satisfies \emph{Property I} if $\cA_{P,G}, \cA_{Q,G}$ and $\cA_{\cK}$ are surjective as morphism of vector spaces.  
\end{definition}

  \begin{theorem}\label{thm: injective C_pq}
        Let $G=C_{pq}$, where $p,q$ are distinct prime numbers and $\cA\in \sO_{G}[\DGA]$. Then $\cA$ as a member of $Vec_{G}^\ast$ is injective if $\cA$ satisfies Property I. \end{theorem}

    \begin{proof}
We use \Cref{defn:inj object} to show $\cA$ is injective. Let us assume that we have a morphism of diagrams of $\DGA$s over $\sO_G$ given by
\[
\xymatrix{
\cB\ar@{^{(}->}[r]^{i}\ar[d]_{\alpha} & \cC\ar@{-->}^{\beta}[dl]\\
\cA
}
\]
where $\cB,\cC\in \sO_G[\DGA]$, $i$ is an inclusion and $\alpha$ is any morphism in this category. We show that there is a map $\beta$ which is an extension of $\alpha$.

Note that at $G/G$ we have a diagram of rational vector spaces, so $\beta(G/G)$ can be constructed easily. Next we construct $\beta(G/P)$. Consider the map $\beta': \cC(G/P)\to \cA(G/G) $ given by the composition $\cC(G/P) \to \cC(G/G) \xrightarrow{\beta(G/G)} \cA(G/G)$, where the first map is the structure map of the diagram $\cC$.

Now from the hypothesis it follows that as vector spaces $\cA(G/P)=\cA(G/G)\oplus V_P$, where $V_P=\ker \cA_{P,G}$. Now consider the composition $$ \cB(G/P)\xrightarrow{\alpha(G/P)} \cA(G/P)= \cA(G/G)\oplus V_P\xrightarrow{\pi}V_P  $$ and define the map $\theta: \cC(G/P)\to V_P$ extending $\pi \alpha(G/P)$. Define $\beta(G/P):=\beta'\oplus \theta$. By construction, $\beta(G/P)$ is compatible with the structure maps of $\cA$ and $\cC$ at the level $G/G$.

In a similar way we can define the map $\beta(G/Q)$ which satisfies this naturality condition. 

Finally, we construct the map $\beta(G/e)$. Consider the commutative diagram (\Cref{Fig:22}).

\begin{figure}[h] \centerline{%
 \xymatrix{
 & \cC(G/e)\ar[dl]\ar[dr] & & \\
\cC(G/P)\ar[dr]\ar[dd] &  & \cC(G/Q)\ar[dl]\ar[dd] \\
& \cC(G/G)\ar[dd] & & \\
\cA(G/P)\ar[dr] &  & \cA(G/Q)\ar[dl] \\
& \cA(G/G) & &
 }}
\caption{}\label{Fig:22}
\end{figure}

where the maps are either $\beta$ or structure maps. By the universal property of pullback, we have a map $\tilde{\beta}: \cC(G/e) \to \cK$. Also, from the hypothesis as vector spaces we have, $\cA(G/e)=\cK \oplus V_{\cK}$, where $V_{\cK}=\ker \cA_{\cK}$.

Consider the composition map $$\cB(G/e)\xrightarrow{\alpha(G/e)} \cA(G/e)=\cK \oplus V_{\cK}\xrightarrow{\pi_{\cK}} V_{\cK}$$
and define the map $\phi:\cC(G/e)\to V_{\cK}$ which is the extension of $\pi_{\cK}\alpha(G/e)$. Define $\beta(G/e):=\tilde{\beta}\oplus \phi$.

To check $\beta(G/e)$ is indeed part of the natural transformation, we must show that it commutes with the structure maps of $\cA$. Now any structure map $\cA(G/e)\to \cA(G/H)$, where $\{e\}\neq H$ subgroup of $G$, must factors through $\cK$ and therefore, the choice of splitting of $\cA(G/e)$ is not important. Upon examination of the \Cref{Fig:2}.

\begin{figure}[h] \centerline{%
 \xymatrix{
\cC(G/e)\ar[rr]\ar[d]_{\beta(G/e)}\ar[dr]^{\tilde{\beta}} && \cC(G/H)\ar[d]^{\beta(G/H)}\\ 
\cA(G/e)\ar[r]_{\cA_{\cK}} & \cK\ar[r] & \cA(G/H)
 }}
\caption{Naturality of $\beta$}\label{Fig:2}
\end{figure}


we see that $\beta$ commutes with all structure maps between values of the functor for which it is defined. Hence the result follows.

\end{proof}

Next, we show the following. 
\begin{prop}
   Let $G=C_{pq}$, where $p,q$ are distinct prime numbers and $\cA\in \sO_{G}[\DGA]$. If $\cA$ is injective as a member of $Vec^*_G$, then all the structure maps are surjective.
    
\end{prop}
\begin{proof}

Let $\cA$ be injective. 
Then  the injective envelope of $\cA$ is isomorphic to $\cA$ and is given by $\sI(\cA)=\underline{\cI}_G^\ast \oplus \underline{\cI}_H^\ast \oplus \underline{\cI}_K^{\ast} \oplus \underline{\cI}^{\ast}_e$, where $\underline{\cI}_L^{\ast}$ are systems (see Equation \ref{equation:24} and \Cref{sum inj}) corresponding  to the vector spaces $I_L=\cap_{L\subset P}\ker\cA(\hat{e}_{L,P}) $ (here $L\subset P$ implies that $P$ properly contains a conjugate of $L$) for every proper subgroup $L$ of $G$ and $I_G=\cA(G/G)$.

From Equation \ref{equation:24} we have, $\underline{\cI}_X^{\ast}(G/K)= Hom_{\QQ(N(X)/X)}(\QQ(G/X)^K,I_X)$.

Thus we have,  
\[
\underline{\cI}^{\ast}_e\cong
\begin{cases}
    I_e, & \text{at G/e}\\
    0, & \text{ else.}
\end{cases}
\]
As $I_G=\cA(G/G)$ we get,
\[
\underline{\cI}^{\ast}_G\cong
\begin{cases}
     \cA(G/G), & \text{at G/e}\\
    \cA(G/G),  & \text{ at G/H}\\
    \cA(G/G), & \text{ at G/K}\\
    \cA(G/G), & \text{ at G/G}.
\end{cases}
\] 
\[
\underline{\cI}^{\ast}_H\cong
\begin{cases}
     I_H=\ker \cA(\hat{e}_{H,G}) , & \text{at G/e}\\
    I_H= \ker \cA(\hat{e}_{H,G}),  & \text{ at G/H}\\
    0, & \text{ at G/K}\\
    0, & \text{ at G/G}.
\end{cases}
\] 
\[
\underline{\cI}^{\ast}_K\cong
\begin{cases}
     I_K=\ker \cA(\hat{e}_{K,G}), & \text{at G/e}\\
    0,  & \text{ at G/H}\\
    I_K=\ker \cA(\hat{e}_{K,G}), & \text{ at G/K}\\
    0, & \text{ at G/G}.
\end{cases}
\] 
Thus the injective envelope for $\cA$ is given by

\[
\sI(\cA)\cong
\begin{cases}
    \cA(G/G)\oplus \ker \cA(\hat{e}_{H,G}) \oplus \ker \cA(\hat{e}_{K,G})\oplus I_e, & \text{at G/e}\\
   \cA(G/G)\oplus \ker \cA(\hat{e}_{H,G}),  & \text{ at G/H}\\
   \cA(G/G)\oplus \ker \cA(\hat{e}_{K,G}), & \text{ at G/K}\\
    \cA(G/G), & \text{ at G/G}.
\end{cases}
\] 

The structure maps are projections and hence surjective.

\end{proof}

We now consider the following example where the diagram of $\DGA$s is not injective but structure maps are all surjective.  

\begin{exmp}\label{exmp: notinjectiveexample}
Let $p,q$ be distinct primes and $P$ and $Q$ be subgroups of $C_{pq}$ of order $p$ and $q$ respectively. 
    Consider the $C_{pq}$-diagram of $\DGA$s
    
    \[
\cB:=
\begin{cases}
     \wedge(x), & \text{at G/e}\\
    \wedge(y),  & \text{ at G/P}\\
    \wedge(z), & \text{ at G/Q}\\
    \QQ, & \text{ at G/G}.
\end{cases}
\]
     given in \Cref{fig:nonexmp}. Here $\deg x=\deg y=\deg z=3$, and differentials are zero. The structure maps are given by $\cB^{e,P}(x)=y$, $\cB^{e,Q}(x)=z$, $\cB^{P,C_{pq}}(y)=0$ and $\cB^{Q,C_{pq}}(z)=0$ making all the structure maps surjective. One may check using \Cref{sum inj} that $\cB$ is not injective. Note that, $\cB$ does not satisfy Property I.    
    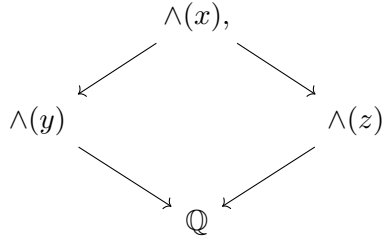
\begin{figure}[h!]
    \centering
\begin{tikzcd}
    & {\wedge(x), }\\
     {\wedge(y)} && {\wedge(z)} \\
     & \QQ\\
     &&&{}
     \arrow[from=1-2, to=2-1]
     \arrow[from=1-2, to=2-3]
     \arrow[from=2-1, to=3-2]
     \arrow[from=2-3, to=3-2]
\end{tikzcd}
\caption{An example of non-injective diagram}
    \label{fig:nonexmp}
\end{figure}

\end{exmp}

\begin{exmp}\label{exmp:1}
Let $T= (S^3\times S^3 \times S^3) \vee (S^5 \times S^5)$ with a $C_6$ action given by as follows: 
Let $H=<t>,K=<s>$ be subgroups of $C_6$ order $3,2$ respectively, $H$ acts on $(S^3\times S^3 \times S^3)$ by $t(x,y,z)=(y,z,x)$ and on $S^5\times S^5$ component the action is trivial. $K$ acts trivially on $(S^3\times S^3 \times S^3)$ and on $(S^5\times S^5)$ by $s(a,b)=(b,a)$. 

Thus we get $X^e=X$, $X^H=S^3 \vee (S^5\times S^5)$, $X^K=(S^3\times S^3 \times S^3)\vee S^5$ and $X^G=S^3 \vee S^5$.

The cohomology diagram is given by 
\[
\cA=
\begin{cases}
     \wedge (x_1,x_2,x_3,y_1,y_2)/D , & \text{at G/e}\\
    \wedge(x,y_1,y_2)/A,  & \text{ at G/H}\\
    \wedge(x_1,x_2,x_3,y)/B, & \text{ at G/K}\\
    \wedge(x,y)/C, & \text{ at G/G},
\end{cases}
\]

where $D=\langle x_iy_j~|~i=1,2,3 \text{ and } j=1,2\rangle$, 
$A=\langle xy_i~|~i=1,2\rangle$, 
$B=\langle x_iy~|~i=1,2,3 \rangle$, $C=\langle xy \rangle$.

The $x_i, y_j$ correspond to each copy of $S^3,S^5$ in $X^e$ and $x,y$ correspond to the copies of $S^3,S^5$ in $X^G$ respectively. The structure maps are defined in the following fashion: 
$\cA_{e,H}([x_i])=[x], \cA_{e,H}([y_i])=[y_i], \cA_{e,K}([x_i])=[x_i], \cA_{e,K}(y_i=y),\cA_{H,G}(][x])=[x], \cA_{H,G}([y_i])=[y], \cA_{K,G}([x_i])=[x], \cA_{K,G}([y])=[y]$. 

Note that the maps $\cA_{H,G}, \cA_{K,G}$ are surjective. Also the pullback of $\cA_{H,G},\cA_{K,G}$ is generated by the classes $([x],[x_i])$ and $([y_i],[y])$ and each of these elements has pre-image namely $[x_i]$ and $[y_i]$ in $\cA(G/e)$. Therefore, $\cA$ satisfies the hypothesis of \Cref{thm: injective C_pq} and hence $\cA$ is injective. Later, we check that $\cA$ also satisfies the hypothesis of \Cref{sum inj}, which also implies injectivity.
\end{exmp}

\begin{prop}
    Let $G=C_{pq}$, where $p,q$ are distinct primes, and $X$ be a $G$ space with the cohomology diagram  $\cA$ being injective. Then, the cohomology diagram of $X^Q$ as a $P$-space (similarly, the cohomology diagram of $X^P$ as a $Q$-space) is also injective. Moreover, the cohomology diagram of $X$ as a $C_p$-space (respectively $C_q$) is also injective.
\end{prop}
\begin{proof}
    The $C_{pq}$ diagram for $\cA$ is given as follows (see \Cref{fig:new}). 
        \begin{figure}[h]
            \centerline{
    \xymatrix{
    \cA(G/e)\ar[d]_{\cA_{e,P}}\ar[r]^{\cA_{e,Q}} & \cA(G/Q)\ar[d]^{\cA_{Q,G}}\\
    \cA(G/P)\ar[r]_{\cA_{P,G}} & \cA(G/G)
    }}
            \caption{$\cA$ as a $C_{pq}$-diagram of $\DGA$s }
            \label{fig:new}
        \end{figure}

    Note that $\cA(G/e)=H^*(X;\QQ), \cA(G/P)=H^*(X^P;\QQ),\cA(G/Q)=H^*(X^Q;\QQ), \cA(G/G)=H^*(X^G;\QQ)$. 
As $\cA$ is injective it follows from \Cref{thm: injective C_pq} that the structure maps are surjective. Now, if $x\in X^Q$ it follows that $q.x=x$ for all $q\in Q$. Also $(X^Q)^P=\{x\in X^Q~|~p.x=x~\text{for all } p\in P\}$. Therefore, as $x\in (X^Q)^P$ for any $g\in P\cup Q$, we have $g.x=x$. In particular, $g.x=x$ for all $g\in G$. Hence it follows $(X^Q)^P=X^G$ and the cohomology diagram for $X^Q$ is given by $\cA_{Q,G}:\cA(G/Q)\to \cA(G/G)$ which is subjective. By \Cref{prop:inj C_p} the result follows.

Next, consider $X$ as a $P=C_p$-space. The cohomology diagram is given by $\cA_{e,P}:H^*(X;\QQ)\to \H^*(X^P;\QQ)$, which is surjective. Hence the result follows. 
\end{proof}

\section{ Combination theorems}\label{section: combnation theorem}

In this section, we prove the combination theorems as promised in \Cref{section:intro}. We start with the following definitions.

\begin{definition}[Wedge product of $C_p$ and $C_q$ space]\label{defn: equiwedgespace}
Let $X$ (respectively $Y$) be $C_p$-space (resectively $C_q$-space). Let $X^{C_p}$ and $Y^{C_q}$ are corresponding fixed point spaces. We define the $C_{pq}$-space as follows 
\[
X\vee Y:=
\begin{cases}
    X\vee Y, & \text{at G/e}\\
    X^{C_p}\vee Y, & \text{at G/P}\\
    X\vee Y^{C_q}, & \text{at G/Q}\\
    
    X^{C_p}\vee Y^{C_q} 
    ,  & \text{ at G/G}.
\end{cases}
\] 

The structure maps are induced from the structure maps of $X,Y$ as $C_p$ and $C_q$-spaces.

\end{definition}

Using \Cref{defn:wedge of DGAs} one can define the following.

\begin{definition}[Wedge product of diagram of $\DGA$s]\label{defn:wedgeequivar}
    
Let $\cU_1\in \sO_{C_p}[\DGA]$ and $\cU_2\in \sO_{C_q}[\DGA]$s. Then one can combine these diagrams to get a $C_{pq}$-diagram of $\DGA$s, $\cU_1\vee \cU_2$, defined as 

\[
\cU_1\vee \cU_2:=
\begin{cases}
    \cU_1(C_p/e)\vee \cU_2(C_q/e), & \text{at G/e}\\
    \cU_1(C_p/C_p)\vee \cU_2(C_q/e), & \text{at G/P}\\
    \cU_1(C_p/e)\vee \cU_2(C_q/C_q), & \text{at G/Q}\\
    
    \cU_1(C_p/C_p)\vee \cU_2(C_q/C_q) 
    ,  & \text{ at G/G}.
\end{cases}
\] 

where the wedge is defined as in \Cref{defn:wedge of DGAs} and the structure maps are induced from the structure maps of $\cU_1$ and $\cU_2$ see \Cref{fig:eqwedge1}.
\begin{figure}
    \centering
\begin{tikzcd}
    & {\cU_1(C_p/e)\vee \cU_2(C_q/e), }\\
     {\cU_1(C_p/C_p)\vee \cU_2(C_q/e)} && {\cU_1(C_p/e)\vee \cU_2(C_q/C_q)} \\
     & \cU_1(C_p/C_p)\vee \cU_2(C_q/C_q)\\
     &&&{}
     \arrow["\cU_1^{e,C_p}\vee id"', from=1-2, to=2-1]
     \arrow["id \vee \cU_2^{e,C_q}",from=1-2, to=2-3]
     \arrow["id \vee \cU_2^{e,C_q}"',from=2-1, to=3-2]
     \arrow["\cU_1^{e,C_p}\vee id",from=2-3, to=3-2]
\end{tikzcd}
\caption{Diagram for $\cU_1\vee \cU_2$}
    \label{fig:eqwedge1}
\end{figure}

\end{definition}

\begin{prop}\label{prop: combination injective}
    Let $\cU_1,\cU_2$ be injective diagram of $\DGA$s over the groups $C_p,C_q$ respectively. Then $\cU_1\vee \cU_2$ is injective as a member of $Vec^*_{C_{pq}}$. 
\end{prop}
\begin{proof}
We have the following commutative diagram (see \Cref{fig:eqwedge1}) for $\cU_1\vee \cU_2$ using \Cref{defn:wedgeequivar}. 

We show that $\cU_1\vee \cU_2$ satisfies Property I, and the proof follows from \Cref{thm: injective C_pq}. As $\cU_1,\cU_2$ are injective it follows from \Cref{prop:inj C_p} that the structure maps $\cU_1^{e,C_p}$ and $\cU_2^{e,C_q}$ are surjective. Hence the maps $id \vee \cU_2^{e,C_q}$ and $\cU_1^{e,C_p}\vee id$ are surjections. 

The pullback $\DGA$ for the maps $id \vee \cU_2^{e,C_q}$ and $id \vee \cU_1^{e,C_p}$ is a subset of $\cU_1(C_p/C_p)\vee \cU_2(C_q/e) \times \cU_1(C_p/e)\vee \cU_2(C_q/C_q)$ and consists of elements $((a',b),(a,b'))$ such that $$id\vee \cU_2^{e,C_q}(a',b)=\cU_1^{e,C_p}\vee id (a,b')$$ which implies $a'=\cU_1^{e,C_p}(a)$ and $b'=\cU_2^{e,C_q}(b)$.
Hence the elements are 
of type $((\cU_1^{e,C_p}(a),b),(a,\cU_2^{e,C_q}(b)))$ where $a\in \cU_1(C_p/e)$, $b\in \cU_2(C_q/e)$. For such an element we have a pre-image $(a,b)\in \cU_1(C_p/e)\vee \cU_2(C_q/e)$. Hence $\cU_1\vee \cU_2$ satisfies Property I and hence injective.

\end{proof}

\begin{theorem}\label{thm:combmain}
    Let $\cA_1,\cA_2$ be injective diagrams of graded algebras over $C_p$ and $C_q$, respectively. If the structure maps of $\cA_1$ and $\cA_2$ are retract then the minimal model of $\cA_1\vee \cA_2$
    is level-wise minimal.
    \end{theorem}
\begin{proof}
    First note that $\cA_1\vee \cA_2$ is injective due to \Cref{prop: combination injective}, so we do not need to take the injective envelope of $\cA_1\vee \cA_2$ to compute its minimal model. 
    Let $\rho_1:\sM_1\to \cA_1$ and $\rho_2:\sM_2\to \cA_2$ be the minimal models. Then using \Cref{Z_P injimpliesminimal} we have up to isomorphism,

\begin{alignat*}{2}
&\begin{aligned} & \sM_1=
    \begin{cases}
       \wedge V & \text{at~} C_p/e\\ 
       \wedge V' & \text{at~} C_p/C_p
    \end{cases}
    \end{aligned}
    & \hskip 6em
    &\begin{aligned} & \sM_2=
    \begin{cases}
       \wedge W & \text{at~} C_q/e\\ 
       \wedge W' & \text{at~} C_q/C_q
    \end{cases}
     \end{aligned}
\end{alignat*}

    where the structure maps $\sM_1^{e,C_p}:\wedge V\to \wedge V' , \sM_2^{e,C_q}:\wedge W\to \wedge W'$ are surjections (infact they are retractions follows from the proof of \Cref{Z_P injimpliesminimal}) and $\rho_1(C_p/e):\wedge V \to \cA_1(C_p/e)$, $\rho_1(C_p/C_p):\wedge V' \to \cA_1(C_p/C_p)$, $\rho_2(C_q/e):\wedge W \to \cA_2(C_q/e)$, $\rho_2(C_q/C_q):\wedge W' \to \cA_2(C_q/C_q)$ are the non-equivariant minimal models (due to \Cref{Z_P injimpliesminimal}).

    It is not hard to check that the following \Cref{fig: M1M2andA} commutes, where the red arrows are induced from $\rho_1,\rho_2$. Indeed, the diagram \Cref{fig: M1M2andA} can be thought of wedge of two diagrams $\rho_1:\sM_1\to \cA_1$ and $\rho_2:\sM_2\to \cA_2$.   

\begin{figure}[h!]
    \centering
\begin{tikzcd}
    & {\wedge V \vee \wedge W}\\
    {\wedge V' \vee \wedge W} && {\wedge V \vee \wedge W'}\\
    & {\wedge V'\vee \wedge W'}\\
    & {\cA_1(C_p/e)\vee \cA_2(C_q/e), }\\
     {\cA_1(C_p/C_p)\vee \cA_2(C_q/e)} && {\cA_1(C_p/e)\vee \cA_2(C_q/C_q)} \\
     & \cA_1(C_p/C_p)\vee \cA_2(C_q/C_q)\\
     &&&{}
     \arrow["\sM_1^{e,C_p}\vee id"', from=1-2, to=2-1]
     \arrow["id \vee \sM_2^{e,C_q}",from=1-2, to=2-3]
     \arrow["id \vee \sM_2^{e,C_q}"',from=2-1, to=3-2]
     \arrow["\sM_1^{e,C_p}\vee id",from=2-3, to=3-2]
     \arrow["\cA_1^{e,C_p}\vee id"',from=4-2, to=5-1]
     \arrow["id \vee\cA_2^{e,C_q}",from=4-2, to=5-3]
     \arrow["id \vee \cA_2^{e,C_q}"',from=5-1, to=6-2]
     \arrow["\cA_1^{e,C_p}\vee id",from=5-3, to=6-2]
     \arrow[bend left=60,swap, red, from=1-2, to=4-2]
     \arrow[red, from=2-1, to=5-1]
     \arrow[bend right=80,swap,red, from=3-2, to=6-2]
     \arrow[red, from=2-3, to=5-3]
\end{tikzcd}
\caption{Top square indicates $\sM_1\vee \sM_2$, bottom square indicates $\cA_1\vee \cA_2$, red arrows are induced from $\rho_1,\rho_2$}
    \label{fig: M1M2andA}
\end{figure}
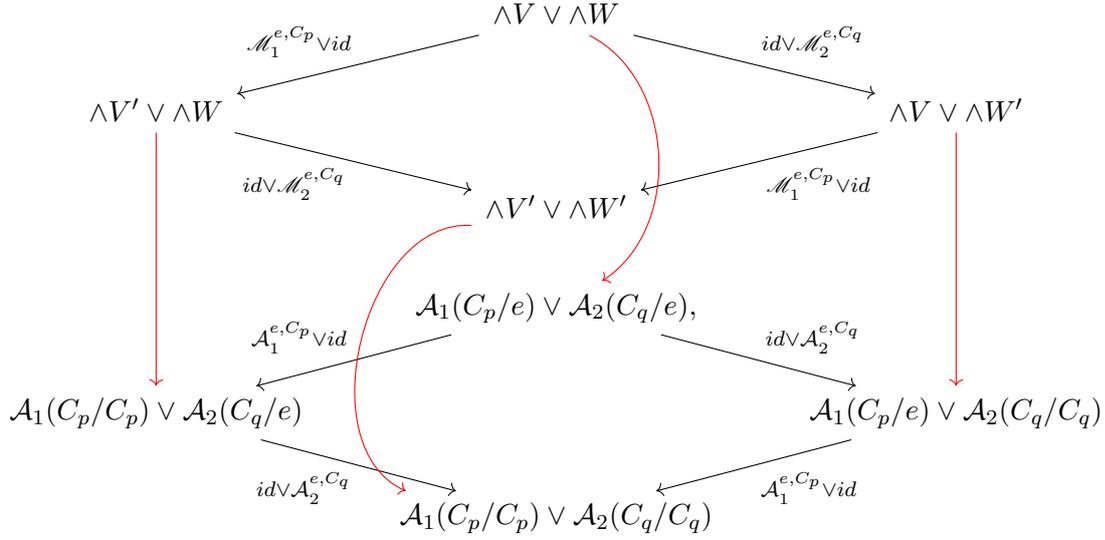

Therefore, the injective diagram (due to \Cref{prop: combination injective}) of $\DGA$s $\sM_1\vee \sM_2$, $\cA_1\vee \cA_2$ are weakly equivalent. As the structure maps of $\sM_1,\sM_2$ are surjections it follows from \Cref{lemma: minimal surjective} that $\pi_V:V\to V'$ and $\pi_W:W\to W'$ are also surjections. 

The minimal model for the wedge sum $\wedge V \vee \wedge W$ is given by the DGA
\[
\wedge(V \oplus W \oplus K_{V,W}),
\]
where $K_{V,W} = \bigoplus_{i \geq 1} K^i_{V,W}$ is a graded vector space constructed inductively as follows:

\begin{itemize}
    \item \textbf{Step 1:} For each pair $v \in V$ and $w \in W$ such that $d(vw) = 0$ in $\wedge(V \oplus W)$, introduce a generator $k \in K^1_{V,W}$ with differential $d(k) = vw$.
    
    \item \textbf{Step 2:} For each $k \in K^1_{V,W}$ and $v \in V$ (respectively, $w \in W$), if $d(kv) = 0$ (respectively, $d(kw) = 0$) in $\wedge(V \oplus W \oplus K^1_{V,W})$, then introduce generators $k_1$, $k_2 \in K^2_{V,W}$ with $d(k_1) = kv$ and $d(k_2) = kw$.
    
    \item \textbf{Inductive Step:} Continue this process: for each $i \geq 1$, if elements in $K^i_{V,W}$ have products with elements in $V$ or $W$ that are closed, introduce generators in $K^{i+1}_{V,W}$ whose differentials kill those closed elements to get rid of product terms in cohomology.
\end{itemize}

This inductive construction ensures that the resulting DGA is minimal and models the rational homotopy type of the wedge sum.

Similarly, for $\wedge V'\vee \wedge W, \wedge V\vee \wedge W', \wedge V'\vee \wedge W'$ the minimal models are given by $\wedge (V'\oplus W\oplus K_{V',W}), \wedge (V\oplus W'\oplus K_{V,W'}), \wedge (V'\oplus W'\oplus K_{V',W'})$ respectively. Consider the following \Cref{fig: justM} of $\DGA$s, $\sM$, where the structure maps of $\sM$ are induced from the surjections $\pi_V:V\to V'$ and $\pi_W:W\to W'$. 
\begin{figure}[h!]
    \centering
\begin{tikzcd}
    & {\wedge(V \oplus W\oplus K_{V,W}) }\\
     {\wedge (V'\oplus W\oplus K_{V',W})} && {\wedge (V\oplus W'\oplus K_{V,W'})} \\
     & \wedge (V'\oplus W'\oplus K_{V',W'})\\
     &&&{}
     \arrow[ from=1-2, to=2-1]
     \arrow[from=1-2, to=2-3]
     \arrow[from=2-1, to=3-2]
     \arrow[from=2-3, to=3-2]
\end{tikzcd}
\caption{Diagram for $\sM$}
\label{fig: justM}
\end{figure}

Note that the maps $K_{V,W}\to K_{V',W}$ are also surjections. Indeed, if $x\in K^1_{V',W}$ then $dx=a'.b$ where $a'\in V'$ and $b\in W$, with the product $a'b$ a cocycle, then as the map $\wedge V\to \wedge V'$ is a retract (follows from \Cref{Z_P injimpliesminimal}) there is a map $i:\wedge V'\to \wedge V$ an inclusion of $\DGA$s so that one can get a preimage of $a'$, say $i(a')=a\in V$, $d(ab)=d(i(a'b))=id(a'b)=0$ and hence there exists $k\in K^1_{V,W}$ such that $dk=ab$ and such $k$ is a pre-image of $x$. A similar thing happens for the subspaces $K^i_{V',W}$.

This implies that the following \Cref{fig:MMM} commutes. 
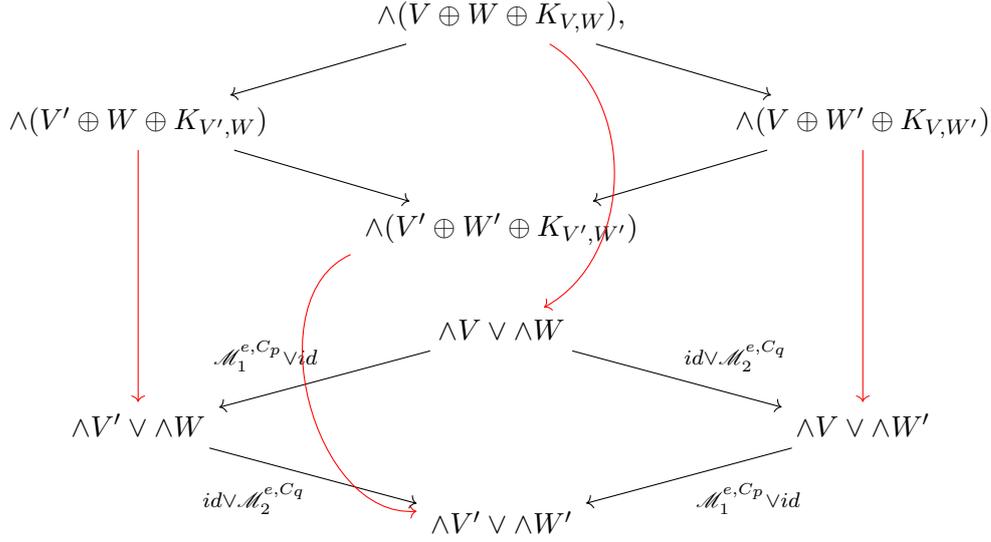
\begin{figure}[h!]
    \centering
\begin{tikzcd}
 & {\wedge(V \oplus W\oplus K_{V,W}), }\\
     {\wedge (V'\oplus W\oplus K_{V',W})} && {\wedge (V\oplus W'\oplus K_{V,W'})} \\
     & \wedge (V'\oplus W'\oplus K_{V',W'})\\
    & {\wedge V \vee \wedge W}\\
    {\wedge V' \vee \wedge W} && {\wedge V \vee \wedge W'}\\
    & {\wedge V'\vee \wedge W'}\\
 &&&{}
     \arrow[ from=1-2, to=2-1]
     \arrow[from=1-2, to=2-3]
     \arrow[from=2-1, to=3-2]
     \arrow[from=2-3, to=3-2]
     \arrow["\sM_1^{e,C_p}\vee id"',,from=4-2, to=5-1]
     \arrow["id \vee \sM_2^{e,C_q}",from=4-2, to=5-3]
     \arrow["id \vee \sM_2^{e,C_q}"',from=5-1, to=6-2]
     \arrow["\sM_1^{e,C_p}\vee id",from=5-3, to=6-2]
     \arrow[bend left=60,swap, red, from=1-2, to=4-2]
     \arrow[red, from=2-1, to=5-1]
     \arrow[bend right=80,swap,red, from=3-2, to=6-2]
     \arrow[red, from=2-3, to=5-3]
\end{tikzcd}
\caption{Top square corresponds to $\sM$ and the bottom square corresponds to $\sM_1\vee \sM_2$}
\label{fig:MMM}
\end{figure}
The red arrows from $V\to V,V'\to V',W\to W, W'\to W'$ are identity, and the vector spaces $K_{V,W},K_{V',W},K_{V,W'},K_{V',W'}$ and the product $vw,v'w,vw',v'w'$ for $v\in V,v'\in V' w\in W,w'\in W'$ are mapped to zero making each red arrows a quasi-isomorphism and hence $\sM\to \sM_1\vee \sM_2$ is a weak equivalence.

We claim that the diagram of $\DGA$s, $\sM$, is the equivariant minimal model for $\cA_1\vee \cA_2$. To prove the claim, we show that the associated diagram of vector spaces (see \Cref{def:associ:vect}) is injective for every $n$, i.e., if $\sM_n=\sM_{n-1}(\underline{V})$ as in the \Cref{def:ele}, then $\underline{V}$ is injective. Note that for $n=0$, we have 
\[
\sM_0=\begin{cases}
    \QQ, \text{~at~} G/e, G/P, G/Q,G/G.\\ 
\end{cases}
\]
Therefore, we have the following \Cref{fig:Mo}. The pull-back corresponding to the maps $\sM_0^{P,G},\sM_0^{Q,G}$ is isomorphic to the diagonal subspace of $\QQ\times \QQ$ and is isomorphic to $\QQ$. Therefore, the map induced by universal property of pullback from $\sM_0(G/e)$ to this pullback space is isomorphism and hence surjection. Hence $\sM_0$ satisfies Property I. 

\begin{figure}[h!]
    \centering
\begin{tikzcd}
    & {\sM_0(G/e)=\QQ }\\
     {\sM_0(G/P)=\QQ} && {\sM_0(G/Q)=\QQ} \\
     & \sM_0(G/G)=\QQ)\\
     &&&{}
     \arrow["\sM_0^{e,P}=id"', from=1-2, to=2-1]
     \arrow["\sM_0^{e,Q}=id", from=1-2, to=2-3]
     \arrow["\sM_0^{P,G}=id"',from=2-1, to=3-2]
     \arrow["\sM_0^{Q,G}=id", from=2-3, to=3-2]
\end{tikzcd}
\caption{Diagram for $\sM$}
\label{fig:Mo}
\end{figure}

Inductively, assume that the associated systems of vector spaces are injective for degree $\leq n$.

We now prove the $n$th stage of induction.
    
Without loss assume that $\pi_V:V\to V'$ and $\pi_W:W\to W'$ are projections. These maps induce projections $\pi_V^n:V_n\to V'_n$ and $\pi_W^n:W_n\to W'_n$ for every $n$, where the vector space $V_n $ (respectively, $V'_n,W_n,W'_n$) consists of vectors homogeneous in degree $n$. As in \Cref{fig: justM} all the arrows are surjections using \Cref{lemma: minimal surjective} it follows that all the arrows in the following \Cref{fig:assoc} are also so, where $h,k,f,g$ are induced naturally from $\pi_V,\pi_W$.
We denote the pull-back vector space for the maps $f,g$ by $R_n$ where the vectors in $R_n$ are $$((v',w, x),(v,w',y))\in V'_n \oplus W_n\oplus (K_{V',W})_n\times  V_n \oplus W'_n\oplus (K_{V,W'})_n $$ such that $v=v',w=w'$ and $f(x)=g(y)=z\in (K_{V',W'})_n$. It follows that there are vectors $a\in V',b\in W'$ so that $d(f(x))=d(g(y))=ab=dz$. Now $V\to V', W\to W'$ are projections so $K_{V',W'}$ embeds inside $K_{V',W},K_{V,W'}\subset K_{V,W}$. 

\begin{figure}[h!]
    \centering
\begin{tikzcd}
    & {V_n \oplus W_n\oplus (K_{V,W})_n }\\
     {V'_n \oplus W_n\oplus (K_{V',W})_n} & R_n& {V_n \oplus W'_n\oplus (K_{V,W'})_n} \\
     & {V'_n \oplus W'_n\oplus (K_{V',W'})_n}\\
     &&&{}
     \arrow[from=2-2, to=2-1]
     \arrow[from=2-2, to=2-3]
     \arrow[dashed,"\gamma",from=1-2, to=2-2]
     \arrow["h"', from=1-2, to=2-1]
     \arrow["k",from=1-2, to=2-3]
     \arrow["f"',from=2-1, to=3-2]
     \arrow[" g",from=2-3, to=3-2]
\end{tikzcd}
\caption{The commutative square indicates the $n$th stage associated diagram of vector spaces for $\sM$, $R_n$ is the pull-back vector space with respect to $f,g$, $\gamma$ is induced by the universal property of pullback} \label{fig:assoc}
\end{figure}

Moreover, $K_{V',W}\cap K_{V,W'}=K_{V',W'}$ as a subspace of $K_{V,W}$. Indeed, as each vector in $K^1_{V',W}$ (respectively, $K^1_{V,W'}$) corresponds to a pair $(v,w')\in V \times W'$ (respectively $V'\times W$) so the vectors in the intersection corresponds to $V'\times W'$.
For the other iterated spaces, a similar thing happens.

Therefore, $x=z+t$ and $y=z+s$, where $t\in K_{V',W}/K_{V',W'}$ and $s\in K_{V,W'}/K_{V',W'}$. Also, by construction $f(t)=0$ it follows that $k(t)=0$, similarly, $g(s)=0=h(s)$. 

Now consider the element $(v',w',z+t+s)\in V_n\oplus W_n\oplus (K_{V,W})_n$. The induced map due to universal property $\gamma:V_n\oplus W_n\oplus (K_{V,W})_n\to R_n$ is given by

\begin{align*}
\gamma(v', w', z + t + s) 
&= \left(h(v', w', z + t + s),\; k(v', w', z + t + s)\right) \notag \\
&= \left((v', w', z + t),\; (v', w', z + s)\right) \notag \\
&= \left((v', w', x),\; (v', w', y)\right).
\end{align*}

This implies the diagram of associated system of vector spaces \Cref{fig:assoc} satisfies Property I and injective for each $n$. This implies $\sM$ is constructed inductively, by elementary extensions of injective diagram of vector spaces for each $n$. Hence $\sM$ is a minimal system. As $\sM \to \sM_1\vee \sM_2\to \cA_1\vee \cA_2$ is a quasi-isomorphism, it follows that $\sM$ is the minimal model of $\cA_1\vee \cA_2$. This completes the proof.

\end{proof}

Using \Cref{thm:combmain} and \Cref{defn: equiwedgespace}, we have the following reslut which gives us a class of $C_{pq}$-cohomology diagram with minimal model level-wise minimal. 
\begin{corollary}\label{cor:main}
    Let $X$ be a $C_p$-space and $Y$ be a $C_q$-space such that the structure maps corresponding to their cohomology diagrams are retracts. Then the cohomology diagram for the $C_{pq}$-space $X\vee Y$ is injective and the minimal model is level wise minimal.  
\end{corollary}

We end this article by proving the following result that gives us a class of examples of $C_{pq}$-formal spaces.

\begin{theorem}\label{thm:main2}
    Let $X,Y$ be equivariantly formal $C_p$ and $C_q$-spaces, $p,q$ distinct primes, then $X\vee Y$ as a $C_{pq}$-space is also formal.
\end{theorem}
\begin{proof}
Let $\cA_1,\cA_2$ denote the injective envelopes of the cohomology diagrams for $X,Y$ respectively with equivariant minimal models $\sM_{\cA_i}$ for $i=1,2$.

    Let $\sM_1$ and $\sM_2$ be the equivariant minimal models of $X$ and $Y$, respectively. As $X,Y$ are equivariantly formal, it follows that the minimal systems $\sM_i$ and $\sM_{\cA_i}$ are weakly equivalent, for $i=1,2$.
    
    Using a similar argument as \Cref{lemma: zigzag weak} and a similar approach used to show the commutativity of \Cref{fig: M1M2andA}, we note that the $C_{pq}$-diagrams of $\DGA$s, $\sM_1\vee \sM_2$ and $\cA_1\vee \cA_2$ are weakly equivalent as 
    $$\sM_1\vee \sM_2 \to (.)\leftarrow{}\sM_{\cA_1}\vee \sM_{\cA_2}\to \cA_1\vee \cA_2.$$

    As $\cA_1\vee \cA_2$ and $\sM_1\vee \sM_2$ are injective using \Cref{prop: combination injective}, and the fact that they are weakly equivalent, it follows that the minimal model of $\sM_1\vee \sM_2$ is weakly equivalent to the injective envelope of $\cA_1\vee \cA_2$, which is itself. Now the minimal model for $X\vee Y$ and the minimal model for $\sM_1\vee \sM_2$ are the same. Hence, the result follows. 

\end{proof}

\begin{exmp}\label{exmp: mainnnn}
Consider the following example as in \Cref{exmp:1}. Let $T= (S^3\times S^3 \times S^3) \vee (S^5 \times S^5)$ with a $C_6$ action given by as follows: 
Let $H=<t>,K=<s>$ be subgroups of $C_6$ order $3,2$ respectively. $H$ acts on $(S^3\times S^3 \times S^3)$ by $t(x,y,z)=(y,z,x)$ and on $S^5\times S^5$ component the action is trivial. $K$ acts trivially on $(S^3\times S^3 \times S^3)$ and on $(S^5\times S^5)$ by $s(a,b)=(b,a)$. 

Thus we get $X^e=X$, $X^H=S^3 \vee (S^5\times S^5)$, $X^K=(S^3\times S^3 \times S^3)\vee S^5$ and $X^G=S^3 \vee S^5$.

The cohomology diagram is given by 
\[
\cA=
\begin{cases}
     \wedge (x_1,x_2,x_3,y_1,y_2)/D , & \text{at G/e}\\
    \wedge(x,y_1,y_2)/A,  & \text{ at G/H}\\
    \wedge(x_1,x_2,x_3,y)/B, & \text{ at G/K}\\
    \wedge(x,y)/C, & \text{ at G/G},
\end{cases}
\]

where $D=\langle x_iy_j~|~i=1,2,3 \text{ and } j=1,2\rangle$, 
$A=\langle xy_i~|~i=1,2\rangle$, 
$B=\langle x_iy~|~i=1,2,3 \rangle$, $C=\langle xy \rangle$.

The $x_i, y_j$ correspond to each copy of $S^3,S^5$ in $X^e$ and $x,y$ correspond to the copies of $S^3,S^5$ in $X^G$ respectively.

We denote the injective envelope for $\cA$ as 
$\sI(\cA)= \underline{\cI}_G^\ast \oplus \underline{\cI}_e^\ast \oplus \underline{\cI}_H^\ast \oplus \underline{\cI}_K^\ast$,

where $\underline{\cI}_X^{\ast}$ is the associated system corresponding to the vector space $I_X=\cap 
_{X\subset Y}\ker (H(\hat{e}_{X,Y}))$ (see \Cref{equation:24}). 
Thus we get

\[
I_X=\begin{cases}
    0, & \text{ X=e}\\
    \wedge(x_{12},x_{13}), & \text{ X=H}\\
    \wedge(y_{12}), & \text{X=K}\\
    \wedge(x,y), & \text{X=G}
\end{cases}
\]

Hence, the corresponding systems are
\[
\underline{\cI}_e^\ast=
\begin{cases}
    0, & \text{ at every subgroup level}
\end{cases}
\]

\[
\underline{\cI}_G^\ast=
\begin{cases}
    I_G=\wedge(x,y)/C, & \text{ at every subgroup level}.
\end{cases}
\] 
\[
\underline{\cI}_H^\ast=
\begin{cases}
     I_H=\wedge(x_{12},x_{13}), & \text{at G/H, G/e}\\
   0 , & \text{else}
\end{cases}
\] 
\[
\underline{\cI}_K^\ast=
\begin{cases}
     I_K=\wedge(y_{12}), & \text{at G/K, G/e}\\
   0 , & \text{else}
\end{cases}
\] 

Thus the injective envelope of the cohomology diagram is 
\[
\underline{\cI}(\cA)=
\begin{cases}
     \wedge(y_{12}) \oplus \wedge(x_{12},x_{13})\oplus \wedge(x,y), & \text{at  G/e}\\
     \wedge(y_{12}) \oplus \wedge(x,y), & \text{at G/H}\\
     \wedge(x_{12},x_{13}) \oplus \wedge(x,y), & \text{at G/K}\\
   \wedge(x,y) , & \text{else}
\end{cases}
\] 

Note that the envelope is isomorphic to $\cA$, hence we conclude that $\cA$ is injective using \Cref{sum inj}.

Next, we compute the minimal model for $\cA\cong \cI(\cA)$. 

\textbf{Computation of $\sM_3$:}
First, consider the diagram, 
    
 \[
\xymatrix{
&\sM_{2}\ar[rd]^{\rho}\ar[d]_{\alpha} \\
\ker (\beta)\ar[r] &\sM'_{2} \ar[r]_{\beta}       & \cA} 
\]

where 
\[
\sM'_2=\sM_2 \otimes (\otimes (I_H \oplus \sum I_H))=
\begin{cases}
     \QQ_0\oplus \QQ_3(x_{12},x_{13},x)\oplus \QQ_4(sx_{12},sx_{13},sx), & \text{at  G/e}\\
     \QQ_0\oplus \QQ_3(x) \oplus \QQ_4(sx), & \text{at G/H}\\
     \QQ_0 \oplus\QQ_3(x_{12},x_{13},x) \oplus \QQ_4(sx_{12},sx_{13},sx) & \text{at G/K}\\
   \QQ_0 \oplus \QQ_3(x) \oplus \QQ_4(sx) , & \text{at G/G.}
\end{cases}
\] 

Therefore, 

\[
(\ker \beta)^4= 
\begin{cases}
   \QQ_4(sx_{12},sx_{13},sx), & \text{at  G/e}\\ 
   \QQ_4(sx), & \text{at G/H}\\
   \QQ_4(sx_{12},sx_{13},sx), & \text{at G/K}\\
   \QQ_4(sx) , & \text{at G/G.}
\end{cases}
\]
Using \Cref{thm: injective C_pq} or \Cref{prop: combination injective}, we conclude that $(\ker \beta)^4$ is injective. So $\underline{V}=H^4(\ker\beta)=(\ker \beta)^4$.

Thus $\sM_3$ is given by 
\[
\sM_3=
\begin{cases}
\wedge(a,b,c), & \text{at G/e}\\
\wedge (c), & \text{at G/H}\\
\wedge(a,b,c), & \text{at G//K} \\
\wedge(c), & \text{at G/G}
\end{cases}
\]
where $|a|=|b|=|c|=3$ with zero differential. The map $\rho $ is the obvious one. 
While computing $\sM_4$ one can check there is no element of degree $5$ in the $\ker (\beta: \sM'_3 \to \cA)$. Hence, we get $\sM_4=\sM_3$.

Similarly, one can check that $\sM_5$ is level-wise minimal and it is given by 
\[
\sM_5=
\begin{cases}
  \wedge(a,b,c,d,e), & \text{at G/e}\\
\wedge (c,d,e), & \text{at G/H}\\
\wedge(a,b,c,e), & \text{at G//K} \\
\wedge(c,e), & \text{at G/G}
\end{cases}
\]

where $\rho$ takes $(a,b,c,d,e)=(x_{12},x_{13},x,y_{12},y)$. Inductively, one can continue this process and get the minimal model. By \Cref{thm:combmain}, we conclude that $\sA$ has the minimal model level-wise minimal. 

Also note that the $C_6$-space $T$ can be written as the equivariant wedge $T_1\vee T_2$, where $T_1=S^3\times S^3\times S^3$, which is a $C_3$-space and $T_2=S^5\vee S^5$, which is a $C_2$-space. From \cite[Example 6.1]{santhanam2023equivariant} it follows that $T_1$ (respectively $T_2$) is $C_3$-equivariantly formal (respectively $C_2$-equivariantly formal).Hence by using \Cref{thm:main2} we conclude that $X$ is $C_{pq}$-equivariantly formal.      
\end{exmp}

\textbf{Acknowledgments:} The author was supported by the TIFR postdoctoral fellowship during this work.

\bibliographystyle{alpha}
\bibliography{reference}

\end{document}